\documentclass[11pt]{article}

\usepackage{amsmath, amssymb, amsthm} 

\usepackage{dsfont}

\usepackage{tikz}

\allowdisplaybreaks
\expandafter\let\expandafter\oldproof\csname\string\proof\endcsname
\let\oldendproof\endproof
\renewenvironment{proof}[1][\proofname]{%
	\oldproof[\bf #1]%
}{\oldendproof}

\allowdisplaybreaks

\theoremstyle{plain}
\newtheorem{theorem}{Theorem}
\newtheorem{lemma}{Lemma}[section]
\newtheorem{claim}[lemma]{Claim}

\RequirePackage[normalem]{ulem} 
\RequirePackage{color}\definecolor{RED}{rgb}{1,0,0}\definecolor{BLUE}{rgb}{0,0,1} 

\title{Hypergraph removal with polynomial bounds}

\author{Lior Gishboliner \thanks{Department of Mathematics, ETH, Z\"urich, Switzerland. Email: lior.gishboliner$@$math.ethz.ch. Research supported by SNSF grant 200021\_196965.} \and Asaf Shapira \thanks{School of Mathematics, Tel Aviv University, Tel Aviv 69978, Israel. Email: asafico$@$tau.ac.il. Supported in part by ISF Grant 1028/16, ERC Consolidator Grant 863438 and NSF-BSF Grant 20196.}}

\begin{document}
\date{}
\maketitle

\begin{abstract}
Given a fixed $k$-uniform hypergraph $F$, the $F$-removal lemma states that
every hypergraph with few copies of $F$ can be made $F$-free by the removal
of few edges. Unfortunately, for general $F$, the constants involved are given by incredibly fast-growing
Ackermann-type functions. It is thus natural to ask for which $F$ one can prove removal lemmas
with polynomial bounds. One trivial case where such bounds can be obtained is when $F$
is $k$-partite. Alon proved that when $k=2$ (i.e. when dealing with graphs), only bipartite graphs
have a polynomial removal lemma. Kohayakawa, Nagle and R\"odl conjectured in 2002 that Alon's result can be
extended to all $k>2$, namely, that the only $k$-graphs $F$ for which the hypergraph removal lemma has polynomial bounds are the trivial cases when $F$ is $k$-partite. In this paper we prove this conjecture.
\end{abstract}

\section{Introduction}\label{sec:intro}

The hypergraph removal lemma is one of the most important results of extremal combinatorics.
It states that for every fixed integer $k$, $k$-uniform hypergraph ($k$-graph for short) $F$ and positive $\varepsilon$,
there is $\delta=\delta(F,\varepsilon) > 0$ so that if $G$ is an $n$-vertex $k$-graph
with at least $\varepsilon n^k$ edge-disjoint\footnote{The lemma's assumption is sometimes stated as $G$ being
$\varepsilon$-far from $F$-freeness, meaning that one should remove at least $\varepsilon n^k$ edges to turn $G$ into an $F$-free hypergraph. It is easy to see that up to constant factors, this notion is equivalent to having $\varepsilon n^k$ edge-disjoint copies of $F$.} copies of $F$, then $G$ contains $\delta n^{v(F)}$ copies
of $F$. This lemma was first conjectured by Erd\H{o}s, Frankl and R\"odl \cite{EFR} as an alternative approach for proving Szemer\'edi's theorem \cite{SzThm}. The quest to proving this lemma, which involved the development of the hypergraph
extension of Szemer\'edi's regularity lemma \cite{Sz}, took more than two decades, culminating
in several proofs, first by Gowers \cite{Gowers07} and R\"odl--Skokan--Nagle--Schacht \cite{NRS,RS3} and later by Tao \cite{Tao}.
For the sake of brevity, we refer the reader to \cite{Rodl14} for more background and references on the subject.

While the hypergraph removal lemma has far-reaching qualitative applications, its main drawback is that it supplies very weak quantitative bounds.
Specifically, for a general $k$-graph $F$, the function $1/\delta(F,\varepsilon)$ grows like the $k^{th}$ Ackermann function.
It is thus natural to ask for which $k$-graphs $F$ one can obtain more sensible bounds.
Further motivation for studying such questions comes from the area of graph property testing \cite{Goldreich1}, where
graph and hypergraph removal lemmas are used to design fast randomized algorithms.

Suppose first that $k=2$. In this case it is easy to see that if $F$ is bipartite then $\delta(F,\varepsilon)$
grows polynomially with $\varepsilon$. Indeed, if $G$ has $\varepsilon n^2$ edge-disjoint copies of $F$ then
it must have at least $\varepsilon n^2$ edges, which implies by the well-known K\"ov\'ari--S\'os--Tur\'an theorem \cite{KST}, that
$G$ has at least $\mbox{poly}(\varepsilon)n^{v(F)}$ copies of $F$. In the seminal paper of Ruzsa and Szemer\'edi \cite{RSz} in which
they proved the first version of the graph removal lemma, they also proved that when $F$ is the triangle $K_3$, the removal
lemma has a super-polynomial dependence on $\varepsilon$. A highly influential result of Alon \cite{Alon} completed the picture by extending
the result of \cite{RSz} to all non-bipartite graphs $F$.

Moving now to general $k > 2$, it is natural to ask for which $k$-graphs the function $\delta(F,\varepsilon)$ depends
polynomially on $\varepsilon$. Let us say that in this case the {\em $F$-removal lemma is polynomial}.
It is easy to see that like in the case of graphs, the $F$-removal lemma is polynomial whenever $F$
is $k$-partite. This follows from Erd\H{o}s's \cite{Erdos64} well-known hypergraph extension of the K\"ov\'ari--S\'os--Tur\'an theorem.
Motivated by Alon's result \cite{Alon} mentioned above, Kohayakawa, Nagle and R\"odl \cite{KNR} conjectured in 2002 that
the $F$-removal lemma is polynomial if and only if $F$ is $k$-partite. They further proved
that the $F$-removal lemma is not polynomial when $F$ is the complete $k$-graph on $k+1$ vertices.
Alon and the second author \cite{AS} proved that a more general condition guarantees that
the $F$-removal lemma is not polynomial, but fell short of covering all non-$k$-partite $k$-graphs.
In the present paper we complete the picture, by fully resolving the problem of Kohayakawa, Nagle and R\"odl \cite{KNR}.

\begin{theorem}\label{theomain}
For every $k$-graph $F$, the $F$-removal lemma is polynomial if and only if $F$ is $k$-partite.
\end{theorem}

As a related remark, we note that for $k \geq 3$, the analogous problem for the {\em induced} $F$-removal lemma (that is, a characterization of $k$-graphs for which the induced $F$-removal lemma has polynomial bounds) was recently settled in \cite{GT21}, following a nearly-complete characterization given in \cite{AS}.

Before proceeding, let us recall the notion of a {\em core}, which plays an important role in the proof of Theorem \ref{theomain}. Recall that for a pair of $k$-graphs $F_1,F_2$, a homomorphism from $F_1$ to $F_2$ is a map $\varphi : V(F_1) \rightarrow V(F_2)$ such that for every $e \in E(F_1)$ it holds that $\{\varphi(x) : x\in e\} \in E(F_2)$. The {\em core} of a $k$-graph $F$ is the smallest (with respect to the number of vertices) subgraph of $F$ to which there is a homomorphism from $F$. It is not hard to show that the core of $F$ is unique up to isomorphism\footnote{Indeed, suppose that $F_1,F_2$ are both cores of $F$. Then $F_1$ is homomorphic to $F_2$ (by taking a homomorphism from $F$ to $F_2$ and restrincting it to $V(F_1)$) and similarly $F_2$ is homomorphic to $F_1$. Also, by the minimality of a core, both homomorphisms $\varphi : F_1 \rightarrow F_2$ and $\psi : F_2 \rightarrow F_1$ must be surjective. Indeed, if e.g.~$\varphi$ is not surjective, then by composing $\varphi$ with a homomorphism from $F$ to $F_1$, we get a homomorphism from $F$ to a proper subgraph of $F_2$, a contradiction. So $|V(F_1)| = |V(F_2)|$ and $\varphi,\psi$ are in fact bijections. It follows that $F_1,F_2$ are isomorphic.}.
Also, note that the core of a $k$-graph $F$ is a single edge if and only if $F$ is $k$-partite. In particular, if a $k$-graph is not $k$-partite, then neither is its core. We say that $F$ is a core if it is the core of itself.


Alon’s \cite{Alon} approach relies on the fact that the core of every non-bipartite graph has a cycle.
It is then natural to try and prove Theorem \ref{theomain} by finding analogous sub-structures in the core of every non-$k$-partite $k$-graphs. Indeed, this was the approach taken in \cite{AS, KNR}.
The main novelty in this paper, and what allows us to handle all cases of Theorem \ref{theomain}, is that instead of directly inspecting the $k$-graph $F$, we study the properties of a certain graph associated with $F$. More precisely, given a $k$-graph $F=(V,E)$, we consider its {\em $2$-shadow}, which is the graph on the same vertex set $V$ in which
$\{u,v\}$ is an edge if and only if $u,v$ belong to some $e \in E$.
The proof of Theorem \ref{theomain} relies on the two lemmas described \nolinebreak below.



\begin{lemma}\label{lemma1}
	Suppose a $k$-graph $F$ is a core and its $2$-shadow contains an induced cycle of length at least 4. Then
	the $F$-removal lemma is not polynomial.\footnote{The proof of this lemma also works if the $2$-shadow of $F$ contains a triangle $x,y,z$ and $|e \cap \{x,y,z\}| \leq 2$ for every $e \in E(F)$, but we will not require this; in fact, this case follows from Lemma \ref{lem:simplex}.} 
\end{lemma}

Note that this is a generalization of Alon's result mentioned above since the $2$-shadow of every non-bipartite graph $F$ (which
is of course $F$ itself in this case) must contain a cycle. Our second lemma is the following.

\begin{lemma}\label{lemma2}
Suppose a $k$-graph $F$ is a core and its $2$-shadow contains a clique of size $k+1$. Then
the $F$-removal lemma is not polynomial.
\end{lemma}

Note that this is a generalization of the result of Kohayakawa, Nagle and R\"odl \cite{KNR} mentioned above since the $2$-shadow of
the complete $k$-graph on $k+1$ vertices is a clique of size $k+1$.

The proofs of Lemmas \ref{lemma1} and \ref{lemma2} appear in
Section \ref{sec:proofs}, but let us first see why they together allow us to handle all non-$k$-partite $k$-graphs, thus proving Theorem \ref{theomain}.


\begin{proof}[Proof of Theorem \ref{theomain}]
The ``if" part was discussed above.
As for the ``only if" part, suppose $F$ is a $k$-graph which is not $k$-partite and assume first that $F$ is a core.
Let $G$ denote the $2$-shadow of $F$.
If $G$ contains an induced cycle of length at least $4$, then the result follows from Lemma \ref{lemma1}.
Suppose then that $G$ contains no such cycle, implying that $G$ is chordal.
Since $F$ is not $k$-partite, $G$ is not $k$-colorable. Since $G$ is assumed to be chordal, and chordal graphs are well-known to be perfect, this means that $G$ has a clique of size $k+1$. Hence, the result follows from Lemma \ref{lemma2}.

To prove the result when $F$ is not necessarily a core, one just needs to observe that if $F'$ is the core of $F$, then $(i)$ as noted earlier, $F'$ is not $k$-partite, and $(ii)$ since the $F'$ removal lemma is not polynomial (by the previous paragraph), then neither is the $F$-removal lemma (see Claim \ref{claim:reduction to core} for the short proof of this fact).
\end{proof}


\section{Proofs of Lemmas \ref{lemma1} and \ref{lemma2}}\label{sec:proofs}

We start by introducing some recurring notions.
Recall that the {\em $b$-blowup} of a $k$-graph $H=(V,E)$ is the $k$-graph obtained by replacing
every vertex $v\in V$ with a $b$-tuple of vertices $S_v$, and then replacing every edge $e=\{v_1,\ldots,v_k\} \in E$ with
all possible $b^k$ edges $S_{v_1} \times S_{v_2} \times \cdots \times S_{v_k}$. Note that if $H'$ is the $b$-blowup of $H$,
then the map sending $S_v$ to $v$ is a homomorphism from $H'$ to $H$. We will frequently refer to this as the {\em natural}
homomorphism from $H'$ to $H$. We say that a $k$-graph $H$ is {\em homomorphic} to a $k$-graph $F$ if there is a homomorphism from $H$ to $F$. 
We first prove the following assertion, which was used in the proof of Theorem \ref{theomain}.

\begin{claim}\label{claim:reduction to core}
Let $F$ be a $k$-graph and let $C$ be a subgraph of $F$ so that $F$ is homomorphic to $C$. Then, if the $C$-removal lemma is not polynomial, then neither is the $F$-removal lemma.
\end{claim}
\begin{proof}
	Since the $C$-removal lemma is not polynomial, there is a function $\delta : (0,1) \rightarrow (0,1)$ such that $1/\delta(\varepsilon)$ grows faster than any polynomial in $1/\varepsilon$, and such that for every $\varepsilon > 0$ and large enough $n$ there is an $n$-vertex $k$-graph $H_1$ which contains a collection $\mathcal{C}$ of $\varepsilon n^k$ edge-disjoint copies of $C$ but only $\delta n^{v(C)}$ copies of $C$ altogether. Let $H$ be the $v(F)$-blowup of $H_1$. Note that the $v(F)$-blowup of $C$ contains a copy of $F$. Also, copies of $F$ corresponding to different copies of $C$ from $\mathcal{C}$ are edge-disjoint. Hence, $H$ has a collection of $\varepsilon n^k = \varepsilon (v(H)/v(F))^k = \Omega(\varepsilon \cdot v(H)^k) = \varepsilon' v(H)^k$ edge-disjoint copies of $F$, for a suitable $\varepsilon' = \Omega(\varepsilon)$. Let us bound the total number of copies of $F$ in $H$. Since $C$ is a subgraph of $F$, each copy of $F$ must contain a copy of $C$.
Let $\varphi : V(H) \rightarrow V(H_1)$ be the natural homomorphism from $H$ to $H_1$ (as defined above).
	For each copy $C'$ of $C$ in $H$, consider the subgraph $\varphi(C')$ of $H_1$. The number of copies $C'$ of $C$ with $v(\varphi(C')) < v(C)$ is at most $v(F)^{v(C)} \cdot O(n^{v(C) - 1}) \leq \delta n^{v(C)}$, provided that $n$ is large enough. The number of copies $C'$ of $C$ with $\varphi(C') \cong C$ is at most $v(F)^{v(C)} \cdot \delta n^{v(C)} = O(\delta n^{v(C)})$, because $H_1$ contains at most $\delta n^{v(C)}$ copies of $C$. So in total, $H$ contains at most $O(\delta n^{v(C)})$ copies of $C$. This means that $H$ contains at most $O(\delta n^{v(C)}) \cdot v(H)^{v(F) - v(C)} =
	O(\delta \cdot v(H)^{v(F)}) = \delta' v(H)^{v(F)}$ copies of $F$, for a suitable $\delta' = O(\delta)$. Note that $1/\delta'$ is super-polynomial in $1/\varepsilon'$. This shows that the $F$-removal lemma is not polynomial.
\end{proof}

Since the core of $F$ satisfies the properties of $C$ in the above claim, it indeed establishes the assertion which we used when proving
Theorem \ref{theomain}, namely that it suffices to prove the theorem when $F$ is a core.

It thus remains to prove Lemmas \ref{lemma1} and \ref{lemma2}. We begin preparing these proofs with some auxiliary lemmas.
The following is a key property of cores that we will use in this section.

\begin{claim}\label{claim:core}
Let $F$ be a core $k$-graph, let $H$ be a $k$-graph, and let $\varphi : H \rightarrow F$ be a homomorphism. Then for every copy $F'$ of $F$ in $H$, the map $\varphi_{|{V(F')}}$ is an isomorphism.
\end{claim}

\begin{proof}
We first observe that every homomorphism from a core $F$ to itself is an isomorphism.
Indeed, by definition, $F$ is the core of itself, meaning that there is no homomorphism from $F$ to a subgraph $F_0$ of $F$ with $V(F_0) \subsetneq V(F)$. Hence, every homomorphism from $F$ to itself is a bijection, and hence an isomorphism. 
The assertion of the claim now follows from the fact that $\varphi_{|{V(F')}}$ is a homomorphism from $F'$ (which is a copy of $F$) to $F$.
\end{proof}

The following definition will play an important role in our proofs. 
Let $F$ be a $k$-graph on vertex-set $[f]$ and let $G$ be an $f$-partite $k$-graph with sides $V_1,\dots,V_f$. A {\em canonical copy} of $F$ in $G$ is a copy consisting of vertices $v_1 \in V_1,\dots,v_f \in V_f$ in which $v_i$ plays the role of $i \in V(F)$ for each $i = 1,\dots,f$. Note that if $G$ is homomorphic to $F$ via the homomorphism mapping $V_i$ to $i$ (for each $i=1,\dots,f$), then $G$ every copy of $F$ in $G$ is canonical; this follows from Claim \ref{claim:core}.

We now describe our approach for proving Lemma \ref{lemma1} (the approach for Lemma \ref{lemma2} is similar).
Let $I \subseteq V(F)$ be a set of vertices so that the $2$-shadow of $F$ induced on $I$ is a cycle $C_t$, $t \geq 4$. Then $|I \cap e| \leq 2$ for every $e \in E(F)$. 
We first use a construction from \cite{Alon}, giving a $t$-partite graph which consists of many edge-disjoint canonical copies of $C_t$, yet contains only few canonical copies of $C_t$ altogether.
The second step is then to extend the graph thus constructed into a $k$-graph containing many edge-disjoint copies of $F$
yet few copies of $F$. The following lemma will help us in performing this extension.
For $\ell \geq 1$, two sets are called {\em $\ell$-disjoint} if their intersection has size at most $\ell-1$. Two subgraphs of a hypergraph are called $\ell$-disjoint if their vertex-sets are $\ell$-disjoint.
In what follows, when considering an $s$-partite hypergraph with parts $V_1,\dots,V_s$, we will refer to the edges as sets or $s$-tuples, interchangeably. Moreover, we will use both set notation and $s$-tuple notation. For example, for $F \in V_1 \times \dots \times V_s$, we write $F(i)$ for the $i$'th coordinate of $F$; and for $F_1,F_2 \in V_1 \times \dots \times V_s$, we write $F_1 \cap F_2$ for the intersection of $F_1,F_2$ as sets. 
	
	\begin{lemma}\label{lem:extension}
		Let $r,s,k,\ell \geq 0$ satisfy $k \geq \ell$ and $r \geq k-\ell$. Let
		$V_1,\dots,V_s,V_{s+1},\dots,V_{s+r}$ be pairwise-disjoint sets of size $n$ each. Let $\mathcal{S} \subseteq V_1 \times \dots \times V_s$ be a family of $\ell$-disjoint sets. Then there is a family $\mathcal{F} \subseteq V_1 \times \dots \times V_{s+r}$ with the following properties:
		\begin{enumerate}
			\item For every $F \in \mathcal{F}$ it holds that $F|_{V_1 \times \dots \times V_s} \in \mathcal{S}$.
			\item $|\mathcal{F}| = \Omega_{r,s,k}(|\mathcal{S}| n^{k-\ell})$.
			\item For every pair of distinct $F_1,F_2 \in \mathcal{F}$, if $|F_1 \cap F_2| \geq k$ then $$\#\{s+1 \leq i \leq s+r : F_1(i) = \nolinebreak F_2(i)\} \leq k-\ell-1$$
		\end{enumerate}
	\end{lemma}
	\begin{proof}
	We construct the family $\mathcal{F}$ as follows. For each $S \in \mathcal{S}$ and each $r$-tuple $A \in V_{s+1} \times \dots \times V_{s+r}$, add $S \cup A$ to $\mathcal{F}$ with probability $1/(Cn^{r-k+\ell})$ and independently, where $C$ is a large constant to be chosen later. Item 1 is satisfied by definition. Let us estimate the number of pairs $F_1,F_2 \in \mathcal{F}$ violating Item 3; denote this number by $B$. We claim that 
	\begin{equation}\label{eq:k-disjoint violations}
		\mathbb{E}[B] = O_{s,r,k}\left(\frac{1}{C^2} \right) \cdot |\mathcal{S}| \cdot n^{k-\ell}.
	\end{equation}
	To this end, suppose that $F_1,F_2 \in \mathcal{F}$ violate Item 3, and write
	$F_1 = S_1 \cup A_1$ and $F_2 = S_2 \cup A_2$, where $S_1,S_2 \in \mathcal{F}$ and $A_1,A_2 \in V_{s+1} \times \dots \times V_{s+r}$.
	Suppose first that $S_1 = S_2$. Then there are $|\mathcal{S}|$ choices for $S_1,S_2$. Also, to violate Item 3, it must hold that $|A_1 \cap A_2| \geq k-\ell$. The number of choices of $A_1,A_2 \in V_{s+1} \times \dots \times V_{s+r}$ with $|A_1 \cap A_2| \geq k-\ell$ is at most $n^{r} \cdot \binom{r}{k-\ell} \cdot n^{r-k+\ell}$. Finally, the probability that $F_1,F_2 \in \mathcal{F}$ is $1/(Cn^{r-k+\ell})^2$. Hence, the expected number of violations of this type (i.e., with $S_1 = S_2$) is at most $|\mathcal{S}| \cdot n^{r} \cdot \binom{r}{k-\ell} \cdot n^{r-k+\ell} \cdot 1/(Cn^{r-k+\ell})^2 = O_{s,r,k}\left(\frac{1}{C^2} \right) \cdot |\mathcal{S}| \cdot n^{k-\ell}$.
	
	Now consider the case that $S_1 \neq S_2$, and put $t := |S_1 \cap S_2|$. 
	As the sets in $\mathcal{S}$ are pairwise $\ell$-disjoint, we have $t \leq \ell-1$. 
	Also, the number of choices for $S_1,S_2 \in \mathcal{S}$ with $|S_1 \cap S_2| = t$ is at most $|\mathcal{S}| \cdot \binom{s}{t} \cdot n^{\ell-t}$, again using that the sets in $\mathcal{S}$ are pairwise $\ell$-disjoint.
	In order for $F_1,F_2$ to violate Item 3, we must have $|A_1 \cap A_2| \geq k-t$. The number of choices for $A_1,A_2 \in V_{s+1} \times \dots \times V_{s+r}$ with $|A_1 \cap A_2| \geq k-t$ is at most $n^{r} \cdot \binom{r}{k-t} \cdot n^{r-k+t}$. Finally, as before, the probability that $F_1,F_2 \in \mathcal{F}$ is $1/(Cn^{r-k+\ell})^2$. Hence, the expected number of violations of this type (i.e., with $S_1 \neq S_2$) is at most
	$$
	\sum_{t=0}^{\ell-1} \left[ |\mathcal{S}| \cdot \binom{s}{t} \cdot n^{\ell-t} \cdot n^{r} \cdot \binom{r}{k-t} \cdot n^{r-k+t} \cdot \left(\frac{1}{Cn^{r-k+\ell}}\right)^2 \right] = 
	O_{s,r,k}\left(\frac{1}{C^2} \right) \cdot |\mathcal{S}| \cdot n^{k-\ell}.
	$$
	This proves \eqref{eq:k-disjoint violations}.
	Now note that the expected size of $\mathcal{F}$ is $|\mathcal{S}| \cdot n^r \cdot \frac{1}{Cn^{r-k+\ell}} = \frac{1}{C} \cdot |\mathcal{S}| \cdot n^{k-\ell}$. So by choosing $C$ to be large enough (as a function of $s,r,k$), we can guarantee that $\mathbb{E}[|\mathcal{F}| - B] \geq \frac{1}{2C} \cdot |\mathcal{S}| \cdot n^{k-\ell}$. By fixing such a choice of $\mathcal{F}$ and deleting one set $F \in \mathcal{F}$ from each violation, we get the required \nolinebreak conclusion.
\end{proof}
	\noindent
	The following well-known fact is an easy corollary of Lemma \ref{lem:extension}.
	\begin{lemma}\label{lem:design}
		Let $1 \leq k \leq r$, and let $V_1,\dots,V_r$ be pairwise-disjoint sets of size $n$ each. Then there is $\mathcal{F} \subseteq V_1 \times \dots \times V_r$, $|\mathcal{F}| \geq \Omega(n^{k})$, such that the $r$-sets in $\mathcal{F}$ are $k$-disjoint.
	\end{lemma}
	\begin{proof}
		Apply Lemma \ref{lem:extension} with $s = \ell = 0$ and $\mathcal{S} = \{\emptyset\}$.
	\end{proof}
\noindent

The next lemma shows why constructing a $k$-graph with a sublinear number of edge-disjoint copies of $F$ can be boosted
to prove Lemmas \ref{lemma1} and \ref{lemma2}. The lemma makes crucial use of the fact that $F$ is a core.

\begin{lemma}\label{lem:blowup}
Let $F$ be a core $k$-graph, and suppose that for every $\delta > 0$ and large enough $n$, there is an $n$-vertex $k$-graph $H$ which is homomorphic to $F$, has a collection of at least $n^{k-\delta}$ edge-disjoint copies of $F$, but has at most $n^{v(F) - 1}$ copies of $F$ altogether. Then the $F$-removal lemma is not \nolinebreak polynomial.
	\end{lemma}
	\begin{proof}
		Let $\varepsilon > 0$ and let $n$ be large enough.
		Let $m$ be the largest integer satisfying $m^{\delta} \leq 1/\varepsilon$, so that $m \geq (1/\varepsilon)^{1/(2\delta)}$, say. Let $H$ be the $k$-graph guaranteed to exist by the assumption of the lemma, but with $m$ in place of $n$. So $H$ has $m$ vertices, is homomorphic to $F$, contains a collection $\mathcal{F}$ of $m^{k-\delta} \geq \varepsilon m^k$ edge-disjoint copies of $F$, but has at most $m^{v(F) - 1}$ copies of $F$ altogether.

Let $G$ be the $\frac{n}{m}$-blowup of $H$. Each $F' \in \mathcal{F}$ gives rise to $\Omega((\frac{n}{m})^k)$ $k$-disjoint (and hence also edge-disjoint) copies of $F$ in $G$, by Lemma \ref{lem:design} applied with $r = v(F)$ and with $\frac{n}{m}$ in place of $n$. Copies arising from different $F'_1,F'_2 \in \mathcal{F}$ are edge-disjoint, because the copies in $\mathcal{F}$ are edge-disjoint.
Altogether, this gives a collection of $\varepsilon m^k \cdot \Omega((\frac{n}{m})^k) = \Omega(\varepsilon n^k)$ edge-disjoint copies of $F$ in $G$.

Let us upper-bound the total number of copies of $F$ in $G$.
By assumption, there is a homomorphism $\varphi$ from $H$ to $F$.
Let $\psi$ be the ``natural'' homomorphism from $G$ to $H$ (as described in the beginning of this section).
Then $\varphi \circ \psi$ is a homomorphism from $G$ to $F$. By Claim \ref{claim:core}, for every copy $F'$ of $F$ in
$G$ the map $(\varphi \circ \psi)|_{{V(F')}}$ is an isomorphism from $F'$ to $F$.
We claim that this means that $\psi$ maps every copy $F'$ of $F$ in $G$ onto a copy of $F$ in $H$.
Indeed, $\psi|_{{V(F')}}$ must be injective (otherwise $(\varphi \circ \psi)|_{{V(F')}}$ would not be an isomorphism),
and since $\psi|_{{V(F')}}$ must map edges to edges (on account of being a homomorphism) its image must contain a copy of $F$.
We thus see that every copy of $F$ in $G$ must come from the blown-up copies of $F$ in $H$.
But each copy of $F$ in $H$ gives rise to $(\frac{n}{m})^{v(F)}$ copies of $F$ in $G$. Hence, the total number of copies of $F$ in $G$ is at most $$
m^{v(F) - 1} \cdot (n/m)^{v(F)} = n^{v(F)}/m \leq \varepsilon^{1/(2\delta)} \cdot n^{v(F)}\;.
$$
Since $\delta > 0$ is arbitrary, this shows that the $F$-removal lemma is not polynomial. 	
\end{proof}	

	\noindent
	The following result is implicit in \cite{Alon}. For the sake of completeness, we include a proof.
	\begin{lemma}
		\label{lem:graph_RS}
		Let $t \geq 3$. Then for every large enough $n$, there is a $t$-partite graph $G$ with sides $V_1,\dots,V_{t}$, each of size $n$, such that $G$ has a collection of $n^2/e^{O(\sqrt{\log n})} = n^{2-o(1)}$ $2$-disjoint canonical copies of $C_t$, but at most $n^{t-1}$ canonical copies of $C_t$ altogether.
	\end{lemma}
	\begin{proof}
Suppose that the vertices of $C_{t}$ are $1,2,\dots,t$ (appearing in this order along the cycle). 
Take a set $B \subseteq [n/t]$, $|B| \geq n/e^{O\sqrt{\log n}}$, with no non-trivial solution to the linear equation $y_1 + \dots + y_{t-1} = (t-1)y_{t}$ with $y_1,\dots,y_{t} \in B$ (where a solution is trivial if $y_1=y_2=\ldots=y_{t}$). The existence of such a set $B$ is by a simple generalization of Behrend's construction \cite{Behrend} of sets avoiding 3-term arithmetic progressions, see \cite[Lemma 3.1]{Alon}. Take pairwise-disjoint sets $V_1,\dots,V_{t}$ of size $n$ each, and identify each $V_i$ with $[n]$. For each $x \in [n/t]$ and $y \in B$, add to $G$ a canonical copy $S_{x,y}$ of $C_{t}$ on the vertices $v_i = x + (i - 1)y \in V_i$, $i = 1,\dots,t$. Note that $x + (i - 1)y \leq x + (t-1)y \leq n$, so $v_i$ indeed ``fits'' into $V_i = [n]$.
		The copies $S_{x,y}$ (where $x \in [n/t], y \in B$) are 2-disjoint. Indeed, if $S_{x_1,y_1},S_{x_2,y_2}$ intersect in $V_i$ and in $V_j$, then $x_1 + (i - 1)y_1 = x_2 + (i - 1)y_2$ and
		$x_1 + (j - 1)y_1 = x_2 + (j - 1)y_2$, and solving this system of equations gives $x_1 = x_2, y_1 = y_2$. The number of copies $S_{x,y}$ is $\frac{n}{t} \cdot |B| \geq n^2/e^{O\sqrt{\log n}}$.
		
		Let us bound the total number of canonical copies of $C_{t}$ in $G$. Fix a canonical copy with vertices $v_1,\dots,v_{t}$, $v_i \in V_i$. 
		For $1 \leq j \leq t-1$, let $x_j \in [n/t], y_j \in B$ be such that $v_{j},v_{j+1} \in S_{x_j,y_j}$. Similarly, let $x_{t} \in [n/t], y_{t} \in B$ such that $v_1,v_{t} \in S_{x_{t},y_{t}}$. Then we have
		$v_{j+1} - v_{j} = y_j$ for every $1 \leq j \leq t-1$, and
		$v_{t} - v_{1} = (t-1)y_{t}$. So $y_1 + \dots + y_{t-1} = (t-1)y_{t}$. By our choice of $B$, we have $y_1 = \dots = y_{t} =: y$. Now, for each $1 \leq j \leq t-1$ we have $x_j = v_{j+1} - j \cdot y = x_{j+1}$, so $x_1 = \dots = x_{t} =: x$. So we see that the only canonical copies of $C_{t}$ in $G$ are the copies $S_{x,y}$. Their number is at most $n^2 \leq n^{t-1}$, \nolinebreak as \nolinebreak required. 
	\end{proof}

%
	\noindent
	Recall that $K_{s}^{(s-1)}$ is the $(s-1)$-graph with vertices $1,\dots,s$ and all $s$ possible edges. The following construction appears implicitly in \cite{KNR} (see also \cite{AS}). Again, for completeness, we include a proof.
	
	\begin{lemma}
		\label{lem:k-simplex_RS}
		Let $s \geq 3$. For every large enough $n$, there is an $s$-partite $(s-1)$-graph $G$ with sides $V_1,\dots,V_s$, each of size $n$, such that $G$ has a collection of $n^{s-1}/e^{O(\sqrt{\log n})} = n^{s-1-o(1)}$ $(s-1)$-disjoint canonical copies of $K_{s}^{(s-1)}$, but at most $n^{s-1}$ copies of $K_{s}^{(s-1)}$ altogether.
	\end{lemma}
	\begin{proof}
Take a set $B \subseteq [n/s]$, $|B| \geq n/e^{O\sqrt{\log n}}$, with no non-trivial solution to $y_1 + y_2 = 2y_3$, $y_1,y_2,y_3 \in B$. Take pairwise-disjoint sets $V_1,\dots,V_s$ of size $n$ each, and identify each $V_i$ with $[n]$. For each $x_1,\dots,x_{s-2} \in [n/s]$ and $y \in B$, add to $G$ a copy $K_{x_1,\dots,x_{s-2},y}$ of $K_s^{(s-1)}$ on the vertices
$$
x_1 \in V_1, ~~~x_2 \in V_2, ~~~\dots ~~~x_{s-2} \in V_{s-2}, ~~~y+\sum^{s-2}_{i=1}x_i \in V_{s-1}, ~~~2y+\sum^{s-2}_{i=1}x_i \in V_s.
$$
It is easy to see that these copies are $(s-1)$-disjoint, because fixing any $s-1$ of the $s$ coordinates allows to solve for $x_1,\dots,x_{s-2},y$. Also, the number of copies thus placed is $(n/s)^{s-2} \cdot |B| \geq n^{s-1}/e^{O\sqrt{\log n}}$. Let us show that there are no other copies of $K_s^{(s-1)}$ in $G$. This would imply that the total number of copies of $K_s^{(s-1)}$ in $G$ is $(n/s)^{s-2} \cdot |B| \leq n^{s-1}$. So suppose that $v_1 \in V_1,\dots,v_s \in V_s$ form a copy of $K_s^{(s-1)}$. Let $x^{(i)} = (x^{(i)}_1,\dots,x^{(i)}_{s-2}) \in [n/s]^{s-2}$ and $y_i \in B$, $i = 1,2,3$, be such that $\{v_2,\dots,v_s\} \in K_{x^{(1)},y_1}$, $\{v_1,\dots,v_{s-1}\} \in K_{x^{(2)},y_2}$ and $\{v_1,\dots,v_{s-2},v_s\} \in K_{x^{(3)},y_3}$. Then $x^{(2)}_1 = x^{(3)}_1 = v_1$ and
		\begin{equation}\label{eq:x}
		x^{(1)}_j = x^{(2)}_j = x^{(3)}_j = v_j \text{ for every } 2 \leq j \leq s-2.
		\end{equation}
		Also, $v_s - v_{s-1} = y_1$, $v_{s-1} - v_1 = x^{(2)}_2 + \dots + x^{(2)}_{s-2} + y_2$ and $v_s - v_1 = x^{(3)}_2 + \dots + x^{(3)}_{s-2} + 2y_3$. Combining these three equations and using \eqref{eq:x}, we get $y_1 + y_2 = 2y_3$, and so $y_1 = y_2 = y_3 =: y$ by our choice of $B$. Also, $x^{(1)}_1 = v_{s-1} - (v_2 + \dots + v_{s-2} +  y) = x^{(2)}_1$. So $x^{(1)} = x^{(2)} = x^{(3)}$.
	\end{proof}
	We now prove two lemmas, Lemmas \ref{lem:induced cycle} and \ref{lem:simplex}, which imply Lemmas \ref{lemma1} and \ref{lemma2}, respectively.
	Recall that for a $k$-graph $F$ and $2 \leq \ell \leq k$, the {\em $\ell$-shadow} of $F$, denoted $\partial_{\ell}F$, is the $\ell$-graph consisting of all
	$f \in \binom{V(F)}{\ell}$ such that there is $e \in E(F)$ with $f \subseteq e$.
	\begin{lemma}
\label{lem:induced cycle}
Let $k \geq 2$, let $F$ be a core $k$-graph, and suppose that $\partial_2 F$ has an induced cycle of length at least $4$. 
Then for every large enough $n$ there is a $k$-graph $H$ with $v(F) \cdot n$ vertices which is homomorphic to $F$, has a collection of $n^k/e^{O(\sqrt{\log n})} = n^{k-o(1)}$ edge-disjoint copies of $F$, but has at most $n^{v(F) - 1}$ copies of $F$ altogether.
	\end{lemma}
	\begin{proof}
		It will be convenient to write $|V(F)| = t+r$ and assume that $V(F) = [t+r]$, where $(1,2,\dots,t,1)$ is an induced cycle in $\partial_2 F$ and $t \geq 4$. It follows that $|e \cap \{1,\dots,t\}| \leq 2$ for every $e \in E(F)$. 
		Take disjoint sets $V_1,\dots,V_{t+r}$ of size $n$ each.
		Let $G$ be the $t$-partite graph with sides $V_1,\dots,V_t$ given by Lemma \ref{lem:graph_RS}.
		Let $\mathcal{S}$ be a collection of $n^2/e^{O(\sqrt{\log n})}$ $2$-disjoint canonical copies of $C_t$ in $G$.
	    Apply Lemma \ref{lem:extension} to\footnote{Strictly speaking, we apply Lemma \ref{lem:extension} to the vertex-sets of the copies of $C_t$ in $\mathcal{S}$.} $\mathcal{S}$ with $s=t$ and $\ell = 2$ to obtain a family $\mathcal{F} \subseteq V_1 \times \dots \times V_{t+r}$ satisfying Items 1-3 in that lemma.
	    Note that $r \geq k-2 = k-\ell$, because each edge of $F$ contains at most two vertices from $\{1,\dots,t\}$ and hence at least $k-2$ vertices from $\{t+1,\dots,t+r\}$. Therefore, the conditions of Lemma \ref{lem:extension} are satisfied.
	    Define the hypergraph $H$ by placing a canonical copy of $F$ on each $F' \in \mathcal{F}$. We claim that these copies of $F$ are edge-disjoint. Indeed, suppose by contradiction that the copies on $F_1,F_2 \in \mathcal{F}$ share an edge $e$. Then $|F_1 \cap F_2| \geq k$. By Item 3 of Lemma \ref{lem:extension}, we have $\#\{t+1 \leq i \leq t+r : F_1(i) = F_2(i)\} \leq k-3$. This implies that $\#\{1 \leq i \leq t : e \cap V_i \neq \emptyset\} \geq 3$. But this means that in $F$ there is an edge which intersects $\{1,\dots,t\}$ in at least $3$ vertices, a contradiction. So the $F$-copies in $\mathcal{F}$ are indeed edge-disjoint. Their number is $|\mathcal{F}| \geq \Omega(|\mathcal{S}|n^{k-2}) \geq n^k/e^{O(\sqrt{\log n})}$, by Item 2 of Lemma \ref{lem:extension}.

To complete the proof, it remains to show that $H$ has at most $n^{t+r-1}$ copies of $F$.
Observe that $H$ is homomorphic to $F$; indeed, the map $\varphi$ which sends $V_j \mapsto j$, $j = 1,\dots,t+r$, is such a homomorphism.
Let $F^*$ be a copy of $F$ in $H$.
Since $F$ is a core and $\varphi$ is a homomorphism from $H$ to $F$, we can apply Claim \ref{claim:core} to conclude that $F^*$
must have the form $v_1,\dots,v_{t+r}$, with $v_i \in V_i$ playing the role of $i$ for each $i = 1,\dots,t+r$. We claim that $v_1,\dots,v_t$ form a canonical copy of $C_t$ {\bf in\footnote{Note that the subgraph of $\partial_2(F^*)$ induced by $v_1,\dots,v_t$ is a canonical copy of $C_t$ in the $2$-shadow of $H$. The first key point is that this copy of $C_t$ must appear in $G$. Also, note that this fact is trivial if $F^*$ is one of the canonical copies of $F$ we placed in $H$ when defining it. The second key point is that this holds for every copy $F^*$ of $F$ in $H$.}} $G$. To see this, fix any $1 \leq i \leq t$ and let us show that $\{v_i,v_{i+1}\} \in E(G)$, with indices taken modulo $t$. Since $\{i,i+1\}$ is an edge of $\partial_2 F$, there must be an edge $e \in E(F)$ containing $i,i+1$. Then $\{v_a : a \in e\} \in E(F^*) \subseteq E(H) = \bigcup_{F' \in \mathcal{F}}{E(F')}$. Let $F' \in \mathcal{F}$ such that
$\{v_a : a \in e\} \in E(F')$. By Item 1 of Lemma \ref{lem:extension}, we have
$S' := F'|_{V_1 \times \dots \times V_t} \in \mathcal{S}$.
Now, $S'$ is the vertex set of a canonical copy of $C_t$ in $G$, and hence $\{v_i,v_{i+1}\} \in E(G)$, as required. This proves our claim that $v_1,\dots,v_t$ form a canonical copy of $C_t$ in $G$. Summarizing, every copy of $F$ in $H$ contains the vertices of a canonical copy of $C_t$ in $G$. By the guarantees of Lemma \ref{lem:graph_RS}, the number of canonical copies of $C_t$ in $G$ is at most $n^{t-1}$. Hence, the number of copies of $F$ in $H$ is at most $n^{t-1} \cdot n^{r} = n^{t+r-1}$, \nolinebreak as \nolinebreak required.
	\end{proof}
	\begin{lemma}
		\label{lem:simplex}
		Let $k \geq 2$, let $F$ be a core $k$-graph and suppose that there are $3 \leq s \leq k + 1$ and a set
		$I \subseteq V(F)$ such that $(\partial_{s-1} F)[I] \cong K_s^{(s-1)}$ and $|e \cap I| \leq s-1$ for every $e \in E(F)$.
		Then for every large enough $n$ there is a $k$-graph $H$ with $v(F) \cdot n$ vertices which is homomorphic to $F$, has a collection of $n^k/e^{O(\sqrt{\log n})} = n^{k-o(1)}$ edge-disjoint copies of $F$, but has at most $n^{v(F) - 1}$ copies of $F$ altogether.
	\end{lemma}
	\begin{proof}
		The proof is very similar to that of Lemma \ref{lem:induced cycle}. Assume that $I = [s]$, $V(F) = [s+r]$. Take disjoint sets $V_1,\dots,V_{s+r}$ of size $n$ each. Let $G$ be the $s$-partite $(s-1)$-graph with sides $V_1,\dots,V_s$ given by Lemma \ref{lem:k-simplex_RS}. Let $\mathcal{S}$ be a collection of $n^{s-1}/e^{O(\sqrt{\log n})}$ $(s-1)$-disjoint copies of $K_{s}^{(s-1)}$ in $G$. Apply Lemma \ref{lem:extension} to $\mathcal{S}$ with $\ell = s-1$ to obtain a family $\mathcal{F} \subseteq V_1 \times \dots \times V_{s+r}$ satisfying Items 1-3 in that lemma. Define the hypergraph $H$ by placing a canonical copy of $F$ on each $F' \in \mathcal{F}$. These copies of $F$ are edge-disjoint. Indeed, suppose by contradiction that the copies on $F_1,F_2 \in \mathcal{F}$ share an edge $e$. Then $|F_1 \cap F_2| \geq k$, and hence
		$\#\{s+1 \leq i \leq s+r : F_1(i) = F_2(i)\} \leq k-\ell-1 = k-s$ by Item 3 of Lemma \ref{lem:extension}. But then $\#\{1 \leq i \leq s : e \cap V_i \neq \emptyset\} = s$, meaning that there is an edge of $F$ which contains $I = [s]$, a contradiction to the assumption of the lemma. So the $F$-copies in $\mathcal{F}$ are indeed edge-disjoint. Also, $|\mathcal{F}| \geq \Omega(|\mathcal{S}|n^{k-s+1}) \geq n^k/e^{O(\sqrt{\log n})}$, using Item 2 of Lemma \ref{lem:extension}.
		
		The map $V_j \mapsto j$, $j = 1,\dots,s+r$ is a homomorphism from $H$ to $F$. Let us bound the number of copies of $F$ in $H$. By Claim \ref{claim:core}, every copy $F^*$ of $F$ must be of the form $v_1,\dots,v_{s+r}$, with $v_i \in V_i$ playing the role of $i$ for each $i = 1,\dots,s+r$.
		We claim that $v_1,\dots,v_s$ span a copy of $K_s^{(s-1)}$ in $G$.
		So let $J \in \binom{[s]}{s-1}$. Since $(\partial_{s-1} F)[I] \cong K_s^{(s-1)}$, there is an edge $e \in E(F)$ with $J \subseteq e$. Since $F^*$ is a canonical copy of $F$, we have $\{v_i : i \in e\} \in E(F^*) \subseteq E(H) = \bigcup_{F' \in \mathcal{F}}{E(F')}$. Let $F' \in \mathcal{F}$ be such that
		$\{v_i : i \in e\} \in E(F')$. By Item 1 of Lemma \ref{lem:extension}, we have
		$S' := F'|_{V_1 \times \dots \times V_s} \in \mathcal{S}$.
		Now, $S'$ is a canonical copy of $K_s^{(s-1)}$ in $G$, and hence $\{v_i : i \in J\} \in E(G)$, as required. So we see that every copy of $F$ in $H$ contains the vertices of a copy of $K_s^{(s-1)}$ in $G$. By the guarantees of Lemma \ref{lem:graph_RS}, $G$ has at most $n^{s-1}$ copies of $K_s^{(s-1)}$. Hence, $H$ has at most $n^{s-1} \cdot n^{r} = n^{s+r-1}$ copies of $F$, as required.
	\end{proof}
	\noindent
	Observe that Lemma \ref{lemma1} follows by combining Lemmas \ref{lem:blowup} and \ref{lem:induced cycle}. Let us prove Lemma \ref{lemma2}.
	\begin{proof}[Proof of Lemma \ref{lemma2}]
		Let $X$ be a clique of size $k+1$ in $\partial_2 F$. Let $I$ be a smallest subset of $X$ which is not contained in an edge of $F$. Note that $I$ is well-defined (because $X$ itself is not contained in any edge of $F$, as $|X| = k+1$). Also, $|I| \geq 3$ because every pair of vertices in $X$ is contained in some edge, as $X$ is a clique in $\partial_2F$. Put $s = |I|$. Then $(\partial_{s-1}F)[I] \cong K_s^{(s-1)}$ and $|e \cap I| \leq s-1$ for every $e \in E(F)$, by the choice of $I$. Now the assertion of Lemma \ref{lemma2} follows by combining Lemmas \ref{lem:blowup} and \ref{lem:simplex}.
	\end{proof}

\end{document}